\documentclass[12pt]{article}
\usepackage[utf8]{inputenc}
\usepackage{amssymb, amsmath,amsthm,mathabx, mathtools}
\usepackage{bbm}
\usepackage{amsfonts}
\usepackage[left=2.5cm,right=2.5cm,top=2cm,bottom=2cm]{geometry}
\usepackage{xcolor}

\usepackage{hyperref}

\usepackage{tabto}
\TabPositions{2 cm, 4 cm, 6 cm, 8 cm}

\usepackage{tikz-cd}

\usepackage{tikz}
\usetikzlibrary{arrows, automata, positioning}

\usepackage{tocloft}
\usepackage{appendix}

\usepackage{enumerate}

\usepackage{hyperref}

\hypersetup{
  colorlinks   = true, 
  urlcolor     = blue, 
  linkcolor    = blue, 
  citecolor   = blue 
}
 
 \newtheorem{thmA}{Theorem}

\newtheorem{conjA}[thmA]{Conjecture}
 
\newtheorem{theorem}{Theorem}[section]
\newtheorem{corollary}[theorem]{Corollary}
\newtheorem{lemma}[theorem]{Lemma}

\newtheorem{proposition}[theorem]{Proposition}

\newtheorem*{theorem*}{Theorem}

\theoremstyle{definition}

\newtheorem{example}[theorem]{Example}
\newtheorem{remark}[theorem]{Remark}

\newtheorem*{remark*}{Remark}
\newtheorem*{notation*}{Notation}
\newtheorem*{acks*}{Acknowledgements}
\newtheorem*{out*}{Outline}

\renewcommand\leq{\leqslant}
\renewcommand\geq{\geqslant}

\newcommand{\qm}{\operatorname{Q}}
\newcommand{\Hom}{\operatorname{Hom}}

\newcommand{\CC}{\operatorname{C}}

\newcommand{\BB}{\operatorname{B}}
\newcommand{\HHH}{\operatorname{H}}
\newcommand{\SSS}{\operatorname{S}}

\newcommand{\N}{\mathbb{N}}
\newcommand{\Z}{\mathbb{Z}}
\newcommand{\Q}{\mathbb{Q}}
\newcommand{\R}{\mathbb{R}}

\newcommand{\cl}{\operatorname{cl}}
\newcommand{\scl}{\operatorname{scl}}

\begin{document}

\title{Stable commutator length on free $\mathbb{Q}$-groups}
\author{Francesco Fournier-Facio}
\date{\today}
\maketitle

\begin{abstract}
We study stable commutator length on free $\mathbb{Q}$-groups. We prove that every non-identity element has positive stable commutator length, and that the corresponding free group embeds isometrically. We deduce that a non-abelian free $\mathbb{Q}$-group has an infinite-dimensional space of homogeneous quasimorphisms modulo homomorphisms, answering a question of Casals-Ruiz, Garreta, and de la Nuez Gonz{\'a}lez. We conjecture that stable commutator length is rational on free $\mathbb{Q}$-groups. This is connected to the long-standing problem of rationality on surface groups: indeed, we show that free $\mathbb{Q}$-groups contain isometrically embedded copies of non-orientable surface groups.
\end{abstract}

\section{Introduction}

A \emph{quasimorphism} on a group $G$ is a function $\varphi \colon G \to \R$ whose \emph{defect}
\[D(\varphi) \coloneqq \sup\limits_{g, h \in G} |\varphi(g) + \varphi(h) - \varphi(gh)|,\]
is finite. A quasimorphism is \emph{homogeneous} if moreover $\varphi(g^n) = n\varphi(g)$ for all $g \in G, n \in \Z$; the space of homogeneous quasimorphisms on $G$ is denoted by $\qm(G)$.
Quasimorphisms are central objects to the study of groups in relation to other areas of mathematics, in particular bounded cohomology \cite{frigerio}, knot theory \cite{knot}, symplectic geometry \cite{symp} and one-dimensional dynamics \cite{ghys}.

In most of these applications, what really matters is the quotient space $\qm(G)/\Hom(G)$ of homogeneous quasimorphisms modulo (real-valued) homomorphisms, which is sometimes called the space of \emph{non-trivial quasimorphisms}. There are several conditions in the literature that ensure that there are no non-trivial quasimorphisms: for example, satisfying a law \cite{law}.
Here we are interested in the dual property, namely, the existence of a \emph{surjective word map}. An element $w$ in a free group $F_m$ defines a word map $w \colon G^m \to G$, obtained by substituting the basis of $F_m$ with the input tuple in $G^m$. If $w \in [F_m, F_m]$, then the surjectivity of the corresponding word map implies that $G$ is \emph{uniformly perfect}, from which it follows easily that it has no non-trivial quasimorphisms. We prove that this is the only case in which such a criterion holds.

\begin{thmA}[Quasimorphisms]
\label{intro:thm:qm}

There exists a countable group $G$ with following properties:
\begin{itemize}
\item For every $w \in F_m \setminus [F_m, F_m]$, the word map $w \in G^m \to G$ is surjective;
\item The space $\qm(G)/\Hom(G)$ is infinite-dimensional.
\end{itemize}
\end{thmA}

This gives an example of a group with infinite-dimensional second bounded cohomology but non-trivial positive theory, answering positively \cite[Question 9.5]{positivetrees}. We also give an example of a different flavour using Thompson groups (Proposition \ref{prop:thompson}).

\medskip

Recall that $G$ is a \emph{$\Q$-group} if every element admits a unique $k$-th root, for every $k \geq 1$. Such groups were introduced by Baumslag \cite{baumslag}, who was particularly intersted in \emph{free $\Q$-groups} on a set $S$, denoted $F_S^{\Q}$. The structure of free $\Q$-groups was further studied by Myasnikov--Remeslennikov \cite{MR1, MR2}, and recently Jaikin-Zapirain proved that they are residually torsion-free nilpotent \cite{jaikin}. The group in Theorem \ref{intro:thm:qm} can be taken to be a non-abelian free $\Q$-group.

We will prove Theorem \ref{intro:thm:qm} via the dual approach of \emph{stable commutator length}, or $\scl$ for short. Given an element $g \in [G, G]$, we denote by $\cl_G(g)$ the minimal number of commutators $[x, y] : x, y \in G$ whose product equals $g$, and
\[\scl_G(g) \coloneqq \lim\limits_{n \to \infty} \frac{\cl_G(g^n)}{n}.\]
If there exists $k \geq 1$ such that $g^k \in [G, G]$, we set $\scl_G(g) = \scl_G(g^k)/k$, and otherwise we set $\scl_G(g) = \infty$. In a natural way, we can extend the domain of definition of $\scl$ to \emph{chains} on $G$ (Section \ref{s:scl}). This is the subject of a rich theory \cite{calegari}, especially thanks to its many incarnations: algebra (via the above definition), geometry (via efficient fillings of loops by surfaces), and most importantly for us: functional analysis. In this context Bavard Duality relates $\scl$ and quasimorphisms \cite{bavard}, which allows to deduce Theorem \ref{intro:thm:qm} from the following statement, in the spirit of \cite{calewalk}.

\begin{thmA}[Isometry]
\label{intro:thm:scl:embedding}
The embedding $F_S \to F_S^{\Q}$ is isometric for $\scl$.
\end{thmA}

Many landmark results on $\scl$ concern \emph{spectral gaps}. The typical statement is of the form: there exists $\varepsilon > 0$ such that for all $g \in [G, G]$, either $\scl_G(g) \geq \varepsilon$ or $\scl_G(g) = 0$, and in the latter case $g$ has to be of a special form (e.g. torsion, or conjugate to its own inverse) \cite{duncanhowie, calefuji, mcg, sgraags, chenheuer}. In a non-abelian free $\Q$-group we cannot hope for a gap: taking $e \neq g \in [F_S, F_S]$, by Theorem \ref{intro:thm:scl:embedding} we have $\scl_{F_S^{\Q}}(g^{1/k}) = \scl_{F_S}(g)/k$, which is positive and arbitrarily small. Still, vanishing only holds for the identity element.

\begin{thmA}[Positivity]
\label{intro:thm:positivity}

Every non-identity element in $F_S^{\Q}$ has positive $\scl$.
\end{thmA}

Theorems \ref{intro:thm:scl:embedding} and \ref{intro:thm:positivity} hold more generally for $A$-completions of torsion-free non-cyclic hyperbolic groups, where $A$ is a subring of $\Q$ (Corollary \ref{cor:hyp}). The infinite-dimensionality of the space of quasimorphisms follows for $A$-completions of more general groups, such as non-abelian right-angled Artin groups (Corollary \ref{cor:raags}).

\medskip

Another important class of results in $\scl$ is \emph{rationality} theorems \cite{cale:rational, chenfreeprod, chengog}. By analogy with the free group, the following conjecture is natural.

\begin{conjA}[Rationality]
\label{conj:rationality}

$\scl$ is rational on $F_S^{\Q}$.
\end{conjA}

This conjecture is likely to be hard. Indeed, one of the main open problems in $\scl$ is whether it is rational on fundamental groups of closed surfaces \cite[Question 7.5.4]{heuer:survey}. It turns out that Conjecture \ref{conj:rationality} is closely related to this problem.

\begin{thmA}[Surfaces]
\label{intro:thm:surfaces}

For all $m \geq 1$ there is an embedding $K_{2m+1} \to F_{2m}^{\Q}$ that is isometric for $\scl$, where $K_{2m+1}$ denotes the fundamental group of the closed non-orientable surface of demigenus $2m+1$.
\end{thmA}

\begin{out*}
We recall some generalities on $\scl$ in Section \ref{sec:general}. Then in Section \ref{sec:iterated} we introduce rational extensions and iterated rational extensions, and prove some general results about their $\scl$. In Section \ref{sec:agroups} we treat $A$-completions, which include free $A$-groups, and apply the general results to them, proving Theorems \ref{intro:thm:scl:embedding} and \ref{intro:thm:positivity}. In Section \ref{sec:theory} we address \cite[Question 9.5]{positivetrees} and prove Theorem \ref{intro:thm:qm}. In Section \ref{sec:rationality} we discuss Conjecture \ref{conj:rationality} and prove Theorem \ref{intro:thm:surfaces}.
\end{out*}

\begin{acks*}
The author is supported by the Herchel Smith Postdoctoral Fellowship Fund. He thanks the Mathematisches Forschungsinstitut Oberwolfach for the hospitality, and Arman Darbinyan for the company. He thanks Montserrat Casals-Ruiz, Lvzhou Chen, Will Cohen, Alexis Marchand, Doron Puder and Henry Wilton for useful conversations, R{\'e}mi Coulon and Turbo Ho for the inspiration while working on \cite{tarski}, and the anonymous referee for the valuable comments.
\end{acks*}

\begin{notation*}
Subgroups of $\Q$ will play an important role. Hence we reserve $1$ for the whole number, and instead use $e$ to denote the identity element of a general discrete group. We will never explicitly write the group operation on $\Q$, to avoid confusion with the formal sums that appear when dealing with chains.
\end{notation*}

\section{Generalities on scl}
\label{sec:general}
\label{s:scl}

We start by extending the definition of $\scl$ to chains, following Calegari \cite{calegari}. Let $G$ be a discrete group. We denote by $\CC_1(G)$ the space of real-valued chains on $G$, namely the real vector space with basis $G$. The subspace of boundaries is
\[\BB_1(G) \coloneqq \mathrm{span}\{ g + h - gh : g, h \in G\}.\]
The quotient $\CC_1(G)/\BB_1(G)$ is the first real homology group $\HHH_1(G)$.

Consider an integral chain $c = \sum_i g_i$, where $g_i \in G$ (possibly with repetitions). Then $\cl_G(c)$ is the smallest number of commutators whose product is equal to an expression of the form
$\prod_i t_i g_i t_i^{-1}$, where $t_i \in G$. We then define
\[\scl_G\left(\sum_i g_i\right) = \lim\limits_{n \to \infty} \frac{\cl_G\left(\sum_i g_i^n\right)}{n},\]
and in fact the limit is an infimum \cite[Lemma 2.73]{calegari}.
By homogeneity, $\scl$ can be extended to all rational chains, and then by continuity to all real chains. It takes finite values precisely on $\BB_1(G)$.

\begin{lemma}
\label{lem:dirun}

Let $(G_j)_{j \in J}$ be a directed system of groups with colimit $G_J$. Let $c \in \CC_1(G_{j_0})$, which we identify with its image in $\CC_1(G_j), j_0 \leq j \in J$. Then
\[\scl_{G_J}(c) = \inf_{j \in J} \scl_{G_j}(c).\]
\end{lemma}

\begin{proof}
Because $\scl$ is monotone, $\scl_{G_J}(c) \leq \inf \scl_{G_j}(c)$. For the converse inequality, we may assume that $\scl_{G_J}(c) < \infty$. Because $\scl$ extends uniquely from integral chains, we may also assume that $c = \sum_i g_i : g_i \in G_{j_0}$. The limit in the definition of $\scl_{G_J}$ is an infimum, so for all $\varepsilon > 0$ there exists $n$ such that
\[\frac{\cl_{G_J}\left( \sum_i g_i^n \right)}{n} < \scl_{G_J}(c) + \varepsilon.\]
There are only finitely many elements involved in an expression witnessing this inequality, which therefore must eventually hold on $G_j$. This implies that $\scl_{G_j}(c) < \scl_{G_J}(c) + \varepsilon$ eventually holds. Letting $\varepsilon \to 0$, we conclude.
\end{proof}

We denote by $\CC_1^H(G)$ the quotient of $\CC_1(G)$ by the subspace spanned by elements of the form $g^n - ng : g \in G, n \in \Z$; and $hgh^{-1} - g : g, h \in G$. The map $\CC_1(G) \to \HHH_1(G)$ factors through $\CC_1^H(G)$. The kernel of the induced map $\CC_1^H(G) \to \HHH_1(G)$ is denoted by $\BB_1^H(G)$, and is the image of $\BB_1(G)$ in $\CC_1^H(G)$. Since $\scl$ vanishes on the elements we quotiented out, and it is subadditive on rational chains, it induces a seminorm on $\BB_1^H(G)$. We denote by $\SSS(G) \subset \BB_1^H(G)$ the subspace of elements with vanishing $\scl$, so that $\scl$ is a norm on the quotient $\BB_1^H(G)/\SSS(G)$.

\begin{example}
\label{ex:chains:Q}

Let $A$ be a non-trivial subgroup of $\Q$. Then every element of $\CC_1^H(A)$ can be written uniquely as a real multiple of a fixed non-zero $a \in A$. It follows that $\CC_1^H(A) \cong \HHH_1(A)$, and so every non-zero element of $\CC_1^H(A)$ has infinite $\scl_A$.
\end{example}

The above example is very basic, but it will be useful later on. For a less trivial example, $\scl$ is a genuine norm on $\BB_1^H(G)$ whenever $G$ is a free group \cite[Proposition 2.84]{calegari} or even a non-elementary hyperbolic group \cite[Corollary 3.57]{calegari}.

\medskip

The fundamental result relating $\scl$ and quasimorphisms is the following.

\begin{theorem}[Generalised Bavard Duality {\cite[Section 2.6]{calegari}}]
\label{thm:bavard}

There is a natural isomorphism between the dual space of $\BB_1^H(G)/\SSS(G)$ with the $\scl$ norm, and the space $\qm(G)/\Hom(G)$, with the norm $2 D(\cdot)$.
\end{theorem}

The way this duality is realised is by evaluating homogeneous quasimorphisms on chains. It is easy to see that the value only depends on the image of the chain in $\CC_1^H(G)$, and that homomorphisms evaluate trivially on $\BB_1^H(G)$.

From this we obtain the result connecting Theorems \ref{intro:thm:qm} and \ref{intro:thm:scl:embedding} from the introduction. Given a map $G_1 \to G_2$, we say that it is \emph{isometric for $\scl$} if the induced map $\CC_1^H(G_1) \to \CC_1^H(G_2)$ preserves $\scl$.

\begin{corollary}
\label{cor:embedding:restriction}

Let $G_1 \to G_2$ be a map that is isometric for $\scl$. Then for every $\varphi_1 \in \qm(G_1)$ there exists $\varphi_2 \in \qm(G_2)$ with the same defect whose pullback to $G_1$ coincides with $\varphi_1$. In particular, the pullback induces a surjection $\qm(G_2)/\Hom(G_2) \to \qm(G_1)/\Hom(G_1)$.
\end{corollary}

\begin{proof}
Because the map preserves chains with infinite $\scl$, it induces an embedding $\HHH_1(G_1) \to \HHH_1(G_2)$. The dual of first real homology is the space of real-valued homomorphisms, hence we obtain the result in case $\varphi_1$ is a homomorphism.

Now suppose that $\varphi_1$ is not a homomorphism. The map realises $\BB_1^H(G_1)/\SSS(G_1)$ as a closed subspace of $\BB_1^H(G_2)/\SSS(G_2)$ with the $\scl$ norm. Theorem \ref{thm:bavard} interprets $\varphi_1$ as a functional on $\BB_1^H(G_1)/\SSS(G_1)$ with operator norm $2D(\varphi_1)$. Hahn--Banach gives a norm-preserving extension, that is, $\varphi_2 \in \qm(G_2)$ with $D(\varphi_2) = D(\varphi_1)$, whose pullback to $G_1$ is equal to $\varphi_1 + \psi_1$, for some $\psi_1 \in \Hom(G_1)$. By the first paragraph we can extend $\psi_1$ to $\psi_2 \in \Hom(G_2)$, and so $\varphi_2 - \psi_2$ is a defect-preserving extension of $\varphi_1$.
\end{proof}

When $G_1 \to G_2$ is an embedding, this is a statement about \emph{extendability} of quasimorphisms. This problem has received much attention over the past few years, especially in two extreme cases: for hyperbolically embedded subgroups \cite{hullosin, FPS}, and for normal subgroups \cite{extend, hhg}. However, in most of these results, the extension has controlled defect, but not \emph{the same} defect. Corollary \ref{cor:embedding:restriction} can be applied to several examples of isometric embeedings between free groups \cite{calewalk, alexis} or graphs of groups \cite{chenfreeprod, chengog, alexis2}.

\begin{remark}
In the above citations, an embedding is called isometric for $\scl$ if it preserves $\scl$ of $\BB_1^H$: it is not required to preserve elements of infinite $\scl$. Under this weaker assumption, the same argument as Corollary \ref{cor:embedding:restriction} shows that every quasimorphism admits a defect-preserving extension, but only up to modifying it by adding a homomorphism.
\end{remark}

Another corollary of Theorem \ref{thm:bavard} is a criterion for positivity of $\scl$.

\begin{corollary}
\label{cor:bavard}

For all $c \in \BB_1^H(G)$, we have: $\scl_G(c) > 0$ if and only if there exists $\varphi \in \qm(G)$ such that $\varphi(c) > 0$. In particular, if $\qm(G) = \Hom(G)$, then every $c \in \BB_1^H(G)$ has vanishing $\scl$.
\end{corollary}

We end this section with two results on $\scl$ of amalgamated products, which will be used in the proofs of Theorems \ref{intro:thm:scl:embedding} and \ref{intro:thm:positivity}, respectively. The first one is a theorem of Chen and Heuer, specialised to this case.

\begin{theorem}[{\cite[Theorem 4.2]{chenheuer}}]
\label{thm:chenheuer}

Let $P = G_1 *_Z G_2$ be an amalgamated product. Let $c_i \in \CC_1^H(G_i)$. Then
\[\scl_P(c_1 + c_2) = \inf \{ \scl_{G_1}(c_1 + d) + \scl_{G_2}(c_2 - d)\},\]
where $d$ runs over $\CC_1^H(Z)$.
\end{theorem}

The second one is a positivity result. We start by recalling some notions.
A subgroup $Z < G$ is \emph{malnormal} if $Z \cap gZg^{-1} = \{ e \}$ for all $g \in G \setminus Z$. An element $g \in G$ is \emph{chiral} if there is no $n \geq 1$ such that $g^n$ is conjugate to $g^{-n}$.

An action of a group $G$ on a metric space $X$ is \emph{acylindrical} if for all $r \in \N$ there exist $R, N \in \N$ such that
\[d(x, y) \geq R \Rightarrow \#\{ g \in G : \max \{ d(x, gx), d(y, gy) \} \leq r \} \leq N.\]
A group $G$ is \emph{acylindrically hyperbolic} if it admits a non-elementary acylindrical action on a Gromov-hyperbolic space. If $g \in G$ is loxodromic for one such action, we say that $g$ is a \emph{generalised loxodromic} element. We isolate the following well-known fact, as it will be used twice.

\begin{lemma}
\label{lem:tfah}

Let $G$ be a torsion-free acylindrically hyperbolic group, and let $g \in G$ be a generalised loxodromic element. Then $g$ is chiral.
\end{lemma}

\begin{proof}
Suppose otherwise. Let $f \in G, n \geq 1$ be such that $f g^n f^{-1} = g^{-n}$. By \cite[Corollary 6.6]{DGO}, $f$ belongs to the elementary closure of $g$, which is the maximal virtually cyclic subgroup of $G$ containing $g$. But $G$ is torsion-free, and virtually cyclic torsion-free groups are cyclic, in particular abelian, which contradicts the relation $f g^n f^{-1} = g^{-n}$.
\end{proof}

\begin{proposition}
\label{prop:positivity:malnormal}

Let $P = G_1 *_Z G_2$ be an amalgamated product of torsion-free groups. Suppose that $Z$ is malnormal in $G_1$. Then every element in $P$ with vanishing $\scl$ is conjugate into a $G_i$.
\end{proposition}

\begin{proof}
The malnormality condition implies that $P$ acts acylindrically on the Bass--Serre tree $T$: the stabiliser of a path of length $3$ is trivial. Hence by \cite[Theorem 4.2]{autinv}, every chiral loxodromic element has positive $\scl$. This applies to all elements that are not conjugate into a $G_i$, since by Lemma \ref{lem:tfah}, loxodromic elements are automatically chiral.
\end{proof}

\section{Iterated rational extensions}
\label{sec:iterated}
\label{sec:ire}

Let $A$ be a subgroup of $\Q$. Let $G$ be a group, and let $Z < G$ be isomorphic to a subgroup of $A$. The amalgamated product
\[G *_Z A\]
is called a \emph{rational extension} of $G$.

\begin{example}
Let $z \in G$ be an infinite order element and let $p \geq 1$. Then the \emph{root extension} $\langle G, t \mid t^p = z \rangle$ is a rational extension of $G$, where $A = \frac{1}{p}\Z$ and $Z$ denotes the infinite cyclic groups $G > \langle z \rangle \cong \Z < A$.
\end{example}

A rational extension $G *_Z A$ is called \emph{malnormal} if $Z$ is a malnormal subgroup of $G$.

\begin{proposition}
\label{prop:rational}

Let $E = G *_Z A$ be a rational extension of $G$.
\begin{enumerate}
\item The embedding $G \to E$ is isometric for $\scl$.
\item Suppose that the rational extension is malnormal, and $\scl_G(g) > 0$ for all $e \neq g \in G$. Then $\scl_E(g) > 0$ for all $e \neq g \in E$.
\end{enumerate}
\end{proposition}

\begin{proof}
First, we notice that the inclusion $G \to E$ induces an embedding $\HHH_1(G) \to \HHH_1(E)$. In particular, if $c \in \CC_1^H(G)$ has infinite $\scl_G$, then it has infinite $\scl_E$.
Now suppose that $c \in \BB_1^H(G)$, so that $\scl_E(c) \leq \scl_G(c) < \infty$. We write it as $c + 0$, seeing $0$ as an element of $\CC_1^H(A)$. Theorem \ref{thm:chenheuer} applies and gives
\[\scl_E(c) = \inf\{ \scl_{G}(c + d) + \scl_A(d)\}.\]
But Example \ref{ex:chains:Q} shows that $\scl_A(d) = \infty$ unless $d = 0$. Since $\scl_E(c) < \infty$, the infimum must be attained at $d = 0$. This gives the first item.

Now with the assumptions of the second item, $G$ must be torsion-free. Let $e \neq g \in E$; if $g$ is not conjugate into either $G$ or $A$, then $\scl_E(g) > 0$ by Proposition \ref{prop:positivity:malnormal}. If $g \in G$, then $\scl_E(g) = \scl_G(g) > 0$, by the first item, and the same follows for all conjugates. If $g \in A$, then there exists $z \in Z$ such that $g$ and $z$ have a common power. Since $\scl_E(z) = \scl_G(z) > 0$, this implies that $\scl_E(g) > 0$ as well, and the same follows for all conjugates.
\end{proof}

\begin{remark}
\label{rem:alexis}

The first item of Proposition \ref{prop:rational} can be proved alternatively using a result of Marchand \cite[Corollary 4.2]{alexis2}. In the case of amalgamated products, it states that if $P = G_1 *_Z G_2$, where $Z$ is amenable, and $\HHH_2(G_1; \Q) \to \HHH_2(P; \Q)$ is surjective, then the embedding $G_1 \to P$ preserves $\scl$ of boundaries.

Now let $E = G *_Z A$ be a rational extension. As in the proof of Proposition \ref{prop:rational}, it is easy to prove directly that the embedding $G \to E$ preserves elements of infinite $\scl$, so by the above it remains to show that $\HHH_2(G; \Q) \to \HHH_2(E; \Q)$ is surjective. The Mayer--Vietoris sequence gives:
\begin{align*}
\cdots &\to \HHH_2(Z; \Q) \to \HHH_2(G; \Q) \oplus \HHH_2(A; \Q) \to \HHH_2(E; \Q) \\
&\to \HHH_1(Z; \Q) \to \HHH_1(G; \Q) \oplus \HHH_1(A; \Q) \to \cdots
\end{align*}
Because $Z$ and $A$ are locally cyclic, they have homological dimension $1$. Moreover, $\HHH_1(Z; \Q) \to \HHH_1(A; \Q)$ is injective, so $\HHH_2(G; \Q) \to \HHH_2(E; \Q)$ is an isomorphism, and we conclude.
\end{remark}

We now push this process by transfinite induction. An \emph{iterated rational extension} is a directed union of groups indexed by ordinals $(G_j)_{j \leq \alpha}$, such that $G_{j+1}$ is a rational extension of $G_j$, and for a limit ordinal $\beta$, $G_\beta$ is the directed union of its subgroups $(G_j)_{j < \beta}$. If all of the rational extensions involved are malnormal, we call this an \emph{iterated malnormal rational extension}.

\begin{theorem}
\label{thm:iterated}

Let $(G_j)_{j \leq \alpha}$ be an iterated rational extension.
\begin{enumerate}
\item The embedding $G_0 \to G_\alpha$ is isometric for $\scl$.
\item Suppose that the iterated rational extension is malnormal, and $\scl_{G_0}(g) > 0$ for all $e \neq g \in G_0$. Then $\scl_{G_\alpha}(g) > 0$ for all $e \neq g \in G_\alpha$.
\end{enumerate}
\end{theorem}

\begin{proof}
This follows by transfinite induction, using Proposition \ref{prop:rational} for successor ordinals, and Lemma \ref{lem:dirun} for limit ordinals.
\end{proof}

\section{$A$-groups}
\label{sec:agroups}

Let $A$ be an associative ring. An \emph{$A$-group} is a group endowed with a map $G \times A \to G : (g,a) \mapsto g^a$ with the following properties, for all $g, h \in G, a, b \in A$:
\begin{enumerate}
\item $g^0 = e, g^1 = g, e^a = e$;
\item $g^{a+b} = g^ag^b, (g^a)^b = g^{ab}$;
\item $(hgh^{-1})^a = hg^a h^{-1}$;
\item If $gh = hg$, then $(gh)^a = g^a h^a$.
\end{enumerate}
We will only be concerned with $A$-groups when $A$ is a subring of $\Q$. The two extreme cases are $\Z$-groups, which are just groups, and $\Q$-groups, which are groups in which every element admits a unique $k$-th root for every $k \geq 1$.

Given a group $G$, every homomorphism from $G$ to an $A$-group factors through its \emph{$A$-completion}, denoted $G^A$: existence and uniqueness of $G^A$ are proved in \cite[Theorems 1 and 2]{MR1}.
In general, this map needs not be injective, and even when it is, it may not be easy to describe explicitly. We will focus on a case where an explicit description is possible. We are only concerned with subrings of $\Q$, in which case this was already achieved by Baumslag \cite[Sections V--VIII]{baumslag}, but we follow the more concise exposition of Myasnikov--Remeslennikov \cite{MR1, MR2}. Recall that a group is \emph{conjugately separated abelian (CSA)} if all centralisers of all non-identity elements are abelian and malnormal.

\medskip

Let $A$ be a subring of $\Q$. Suppose that $G$ is a torsion-free CSA group. We choose a collection $\mathcal{Z}$ of centralisers that are not already $A$-modules, and such that every centraliser that is not an $A$-module is conjugate to a unique $Z \in \mathcal{Z}$. Because every $Z \in \mathcal{Z}$ is torsion-free, we can define the amalgamated product $G *_Z (Z \otimes A)$. This group is still torsion-free, moreover it is CSA \cite[Theorem 5]{MR2}. Choosing a well-ordering of $\mathcal{Z}$, we apply this process inductively to obtain a group $G^*$, which is again torsion-free and CSA \cite[Lemma 6]{MR2}. Moreover, for every $e \neq g \in G$, if $Z$ is its centraliser in $G$, then its centraliser in $G^*$ is isomorphic to $Z \otimes A$ \cite[Lemma 4]{MR2}.

We define $G^{(0)} \coloneqq G$, and by induction $G^{(n+1)} \coloneqq (G^{(n)})^*$. Then the directed union
\[G = G^{(0)} \to G^{(1)} \to \cdots \to G^A\]
coincides with the $A$-completion of $G$ \cite[Theorem 8]{MR2}.

\medskip

Let us assume moreover that $G$ does not contain a copy of $\Z^2$. For a torsion-free CSA group, this is equivalent to saying that all centralisers are locally cyclic, hence isomorphic to subgroups of $\Q$. Then an amalgamated product $G *_Z (Z \otimes A)$ is a malnormal rational extension, and $G \to G^*$ is an iterated malnormal rational extension. By \cite[Lemmas 4 and 6]{MR2}, the group $G^*$ satisfies the same hypotheses as $G$, and so we can extend the inductive process to the directed union $G^{(0)} \to G^{(1)} \to \cdots \to G^A$. Therefore $G^A$ is an iterated malnormal rational extension of $G_0$, and Theorem \ref{thm:iterated} applies.

\begin{theorem}
\label{thm:Acomp}

Let $A$ be a subring of $\Q$, and let $G$ be a torsion-free CSA group that does not contain a copy of $\Z^2$.
\begin{enumerate}
\item The embedding $G \to G^A$ is isometric for $\scl$.
\item Suppose that $\scl_G(g) > 0$ for all $e \neq g \in G$. Then $\scl_{G^A}(g) > 0$ for all $e \neq g \in G^A$. \qed
\end{enumerate}
\end{theorem}

\begin{corollary}
\label{cor:hyp}

Let $A$ be a subring of $\Q$, and let $G$ be a non-cyclic torsion-free hyperbolic group, or a non-abelian free group (of any rank).
\begin{enumerate}
\item The embedding $G \to G^A$ is isometric for $\scl$.
\item $\scl_{G^A}(g) > 0$ for all $e \neq g \in G^A$.
\item The space $\qm(G^A)/\Hom(G^A)$ is infinite-dimensional.
\end{enumerate}
\end{corollary}

\begin{proof}
With these assumptions, $G$ is CSA and does not contain a copy of $\Z^2$, hence the first item follows from Theorem \ref{thm:Acomp}.

If $G$ is a non-elementary hyperbolic group, then all chiral elements have positive $\scl$ \cite{calefuji}. If $G$ is moreover torsion-free, then all elements are chiral (Lemma \ref{lem:tfah}). So Theorem \ref{thm:Acomp} gives the second item for non-cyclic torsion-free hyperbolic groups, in particular for free groups of finite rank. For free groups of infinite rank, it follows from this, and the fact that every element is contained in a retract isomorphic to $F_2$, hence also in this case all non-identity elements have positive $\scl$.

Non-abelian free groups \cite{brooks} and non-elementary hyperbolic groups \cite{epsfuji} admit an infinite-dimensional space of homogeneous quasimorphisms modulo homomorphisms. Hence the third item follows from the first and Corollary \ref{cor:embedding:restriction}.
\end{proof}

\begin{remark}
The quasimorphism spaces from Corollary \ref{cor:hyp} have uncountable dimension. What is more, they are not separable, when seen as Banach spaces endowed with the defect norm. Indeed, this is true for free groups \cite[Example 2.62]{calegari}, from which it follows for all acylindrically hyperbolic groups via extensions from hyperbolically embedded subgroups \cite{hullosin}, from which it follows for any isometric embedding of such a group by Corollary \ref{cor:embedding:restriction}, in particular to the $A$-completions from Corollary \ref{cor:hyp}.
\end{remark}

When $G = F_S$ is a free group, the $A$-completion $F_S^A$ is called the \emph{free $A$-group on $S$}.

\begin{proof}[Proof of Theorems \ref{intro:thm:scl:embedding} and \ref{intro:thm:positivity}]

If $|S| \leq 1$, then $F_S \to F_S^{\Q}$ is either the identity on the trivial group, or the embedding $\Z \to \Q$, and then the statements obviously hold. If $|S| > 1$, Corollary \ref{cor:hyp} applies.
\end{proof}

The first step towards Corollary \ref{cor:hyp} was the fact that rational extensions preserve $\scl$. Both proofs (Proposition \ref{prop:rational} and Remark \ref{rem:alexis}) rely on the fact that we extend along a locally cyclic abelian group. There are other groups where $A$-completions have been described as iterated extensions by centralisers, for example coherent RAAGs \cite{raags}, but the presence of higher rank centralisers is an obstacle to adapting our approach. Nevertheless, we can obtain the infinite-dimensionality result for such groups as well.

\begin{corollary}
\label{cor:raags}

Let $A$ be a subring of $\Q$. If $G^A$ is an $A$-group that surjects onto a free $A$-group, then $\qm(G^A)/\Hom(G^A)$ is infinite-dimensional.

In particular, this is true for the $A$-completion of a non-abelian right-angled Artin group.
\end{corollary}

\begin{proof}
If $G^A \to F^A$ is a surjection, then the pullback $\qm(F^A)/\Hom(F^A) \to \qm(G^A)/\Hom(G^A)$ is injective, so the first statement follows from Corollary \ref{cor:hyp}.

If $G$ is a non-abelian RAAG, then it retracts onto a parabolic subgroup $F$ that is free of rank $2$. So the universal property of $G^A$ gives induces a map:
\[\begin{tikzcd}
	G & {G^A} \\
	F & {F^A}
	\arrow[from=1-1, to=1-2]
	\arrow[two heads, from=1-1, to=2-1]
	\arrow[from=1-1, to=2-2]
	\arrow[dashed, from=1-2, to=2-2]
	\arrow[from=2-1, to=2-2]
\end{tikzcd}\]
Because the image is a $\Q$-group and contains the image of $F \to F^A$, the induced map must be surjective.
\end{proof}

\section{Positive theory}
\label{sec:theory}

The \emph{positive theory} of a group $G$ is the collection of first-order sentences without negations that hold in $G$. The positive theory of every group contains the positive theory of a non-abelian free group \cite{merzljakov, makanin}; if there are no additional positive sentences, then the group is said to have \emph{trivial positive theory}.

Many non-trivial positive sentences imply that every homogeneous quasimorphism is a homomorphism. This is true for example for groups satisfying a law \cite{law}, uniformly perfect groups, and groups with commuting conjugates \cite{kotschick, spectrum}. In \cite[Section 9.3]{positivetrees}, the authors speculate on a relation between positive theory and second bounded cohomology. In particular, they ask whether there exist groups with infinite-dimensional second bounded cohomology and non-trivial positive theory \cite[Question 9.5]{positivetrees} (the converse \cite[Question 9.4]{positivetrees} is addressed in \cite{tarski}). Note that the space of homogeneous quasimorphisms modulo homomorphisms embeds in the second bounded cohomology \cite[Theorem 2.50]{calegari}. In this section, we give two examples.

\begin{lemma}
\label{lem:wordmap}

Let $G$ be a $\Q$-group. Then for all $w \in F_m \setminus [F_m, F_m]$, the word map $w \colon G^m \to G$ is surjective.
\end{lemma}

\begin{proof}
Write
\[w = x_1^{i_1} \cdots x_m^{i_m} v;\]
where $v \in [F_m, F_m]$. Because $w \notin [F_m, F_m]$, one of the exponents $i_k$ must be non-zero; up to reordering the basis, we may assume that $i_1 \neq 0$. No $w(g, 1, \ldots, 1) = g^{i_1}$, so $w(G^m)$ contains the image of $x \mapsto x^{i_1}$, which is surjective on a $\Q$-group.
\end{proof}

\begin{proof}[Proof of Theorem \ref{intro:thm:qm}]
Let $G$ be the $\Q$-completion of a torsion-free non-cyclic hyperbolic group. Then $G$ satisfies both statements, by Lemma \ref{lem:wordmap} and Corollary \ref{cor:hyp}.
\end{proof}

Surjectivity of a word map is a positive sentence, but not always a non-trivial one. Indeed, there are certain words, called \emph{silly} by Segal \cite[Section 3.1]{segal}, that define a surjective word map on all groups. These are the ones that can be written as $w = x_1^{i_1} \cdots x_m^{i_m} v$, as in the proof of Lemma \ref{lem:wordmap}, where $\mathrm{gcd}(i_1, \ldots, i_m) = 1$.

Taking a non-silly word map (e.g. $w = y^2$) we get a non-trivial positive sentence (e.g. $\forall x \exists y : y^2 = x$). This shows that $\Q$-groups have non-trivial positive theory, and the ones from Theorem \ref{intro:thm:qm} have infinite-dimensional second bounded cohomology, which gives a positive answer to \cite[Question 9.5]{positivetrees}.

\medskip

We also give a different example, which is less natural but more direct. Consider Thompson's group $T$ \cite{CFP}: this is a finitely presented infinite simple group of homeomorphisms of the circle $\R / \Z$. The group $\widetilde{T}$ is the group of homeomorphisms of $\R$ that commute with integer translations and induce an element of $T$ on $\R/\Z$.

\begin{proposition}
\label{prop:thompson}
    The group $\widetilde{T}^{\N}$ has non-trivial positive theory, and $\qm(\widetilde{T}^{\N})/\Hom(\widetilde{T}^{\N})$ is infinite-dimensional.
\end{proposition}

\begin{proof}
The group $\widetilde{T}$ is perfect, and has a one-dimensional space of homogeneous quasimorphisms, spanned by the rotation quasimorphism (see e.g. \cite[Chapter 5]{calegari}). It follows that $\widetilde{T}^{\N}$ has an infinite-dimensional space of homogeneous quasimorphisms, and no non-zero homomorphisms.

Moreover, $T$ is uniformly perfect \cite{T:uniformlyperfect}, and hence so is $T^{\N}$, which in particular has non-trivial positive theory. Now $\widetilde{T}^{\N}$ is a central extension of $T^{\N}$, and hence it also has non-trivial positive theory \cite[Theorem G]{positivetrees}.
\end{proof}

Note that both examples are countable but infinitely generated. We do not know of a finitely generated example.

\section{Rationality}
\label{sec:rationality}

We say that \emph{$\scl$ is rational} on a group $G$, if $\scl_G(c) \in \Q$ for all $c \in \CC_1^H(G)$ with rational coefficients. Rationality on free groups is one of the most influential results on $\scl$ \cite{cale:rational}, which motivated our Conjecture \ref{conj:rationality}. To approach it, one would have to start by showing that malnormal rational extensions of free groups have rational $\scl$. Indeed, this is not only a reasonable first step, but also necessary.

\begin{proposition}
\label{prop:rational:embed:completion}

Let $G$ be a torsion-free CSA group that does not contain a copy of $\Z^2$. Let $z \in G$ be an element that generates its own centraliser. Then every rational extension of the form $G *_{z = a} A$ embeds into $G^{\Q}$, isometrically for $\scl$.

In particular, if $\scl$ is rational on $G^{\Q}$, then it is rational on $G *_{z = a} A$.
\end{proposition}

\begin{proof}
Recall the construction of $G^{\Q}$ from Section \ref{sec:agroups}. The first step (that is, the embedding $G = G^{(0)} \to G^{(1)}$) consists in choosing a set $\mathcal{Z}$ of representatives of conjugacy classes of centralisers in $G$, and performing the single rational extension $G *_Z (Z \otimes \Q)$ where $Z \in \mathcal{Z}$ is the first element in a well-ordering of $\mathcal{Z}$. By assumption $z \in G$ generates its own centraliser, so we may choose it to generate the first centraliser in $\mathcal{Z}$. Hence
\[G \to G *_{\langle z \rangle} (\langle z \rangle \otimes \Q) = G *_{z = 1} \Q\]
is the first step of the iterated rational extension that constructs $G^\Q$. This shows that $G^\Q$ is an iterated rational extension of $G *_{z = 1} \Q$.

Now consider a rational extension $G *_{z = a} A$. Up to replacing $A$ by an isomorphic subgroup of $\Q$, we may assume that $a = 1$. But then
\[G *_{z = 1} A \to (G *_{z = 1} A) *_{A} \Q \cong G *_{z = 1} \Q\]
is a rational extension. Therefore $G^{\Q}$ is an iterated rational extension of $G *_{z = 1} A$. We conclude by Theorem \ref{thm:iterated}.
\end{proof}

It turns out that it suffices to look at root extensions.

\begin{lemma}
\label{lem:roottorational}

Let $G$ be a group and $z \in G$. Suppose that every root extension of the form $\langle G, t \mid t^p = z \rangle$ has rational $\scl$. Then every rational extension of the form $G *_{z = a} A$ has rational $\scl$.
\end{lemma}

\begin{proof}
We start by proving the lemma for $E = G *_{z = 1} \Q$. For all $n \geq 2$, define
\[E_n \coloneqq G *_{z = 1} \left( \Z \left[ \frac{1}{n!} \right] \right) \cong \langle G, t \mid t^{n!} = z \rangle.\]
By assumption, $\scl$ is rational on $E_n$. Moreover,
\[ E_{n+1} \cong E_n *_{\Z \left[ \frac{1}{n!} \right]} \left( \Z \left[ \frac{1}{(n+1)!} \right] \right).\]
This shows that $E_{n+1}$ is a rational extension of $E_n$, and so by Proposition \ref{prop:rational} the embedding $E_n \to E_{n+1}$ is isometric. Finally, $E$ is the directed union of the $E_n$, so by Lemma \ref{lem:dirun} we conclude.

In general, let $E = G *_{z = a} A$, where $a \in A$. Up to changing $A$ to an isomorphic subgroup of $\Q$, we may assume that $a = 1$. Now consider
\[E' = E *_A \Q \cong G *_{z = 1} \Q.\]
It is a rational extension of $E$, so by Proposition \ref{prop:rational} the embedding $E \to E'$ is isometric. Moreover, $\scl$ is rational on $E'$ by the previous paragraph, and so $\scl$ is rational on $E$.
\end{proof}

However, square root extensions of the free group already present a major obstacle. Let $m \geq 1$. Then the following square root extension of $F_{2m}$:
\[K_{2m+1} = \langle a_1, b_1, \ldots, a_m, b_m, t \mid t^2 = [a_1, b_1] \cdots [a_m, b_m] \rangle\]
is the fundamental group of a closed non-orientable surface. Its orientable double cover corresponds to the index-$2$ subgroup
\[\langle a_i, b_i \rangle *_{\prod [a_i, b_i] = \prod t [a_i, b_i] t^{-1}} \langle t a_i t^{-1}, t b_i t^{-1} \rangle,\]
which is the fundamental group of a closed surface of genus $2m$. It follows that the Euler characteristic of the non-orientable surface is $(2 - 4m)/2 = 1 - 2m$, and so its demigenus is $2 - (1 - 2m) = 2m + 1$, hence the notation. When $m  = 1$, we recover the classical presentation of Dyck's surface \cite{dyck}.

\begin{proof}[Proof of Theorem \ref{intro:thm:surfaces}]

$K_{2m+1}$ is a root extension of $F_{2m} = \langle a_1, b_1, \ldots, a_m, b_m \rangle$ over the element $z = [a_1, b_1] \cdots [a_m, b_m]$. Since $z$ generates its own centraliser, Proposition \ref{prop:rational:embed:completion} shows that $K_{2m+1}$ embeds into $F_{2m}^{\Q}$, isometrically for $\scl$.
\end{proof}

In particular, if $\scl$ were rational on $F_{2m}^{\Q}$, then it would be rational on $K_{2m+1}$. Rationality of $\scl$ on surface groups is a major open problem \cite[Question 7.5.4]{heuer:survey}, which motivated many of the recent advances in stable commutator length \cite{chengog, alexis}.

\footnotesize

\bibliographystyle{amsalpha}
\bibliography{ref}

\vspace{0.5cm}

\normalsize

\noindent{\textsc{Department of Pure Mathematics and Mathematical Statistics, University of Cambridge, UK}}

\noindent{\textit{E-mail address:} \texttt{ff373@cam.ac.uk}}

\end{document}